\documentclass[11pt]{article}
\usepackage[percent]{overpic}
\usepackage{amsmath}
\usepackage{amssymb}
\usepackage{amsthm}
\usepackage{array} 
\usepackage[english]{babel}
\usepackage{enumitem}
\usepackage{lmodern}
\usepackage{xfrac}  

\usepackage{graphicx,float,wrapfig, caption}
\usepackage[percent]{overpic}
\usepackage{pgf, tikz, tikz-cd} 
\graphicspath{ {pictures/}{tikz/} }
\usepackage[tmargin=0.75in,bmargin=0.75in,lmargin=0.75in,rmargin=0.75in]{geometry}
\usepackage{mathtools} 
\usetikzlibrary{arrows} 
\usepackage{extarrows}
\usetikzlibrary{positioning}
\usepackage{standalone} 

\usepackage[small,nohug,heads=vee]{diagrams}
\diagramstyle[labelstyle=\scriptstyle]
\usepackage{hyperref}
\hypersetup{
    unicode=false,          
    pdftoolbar=true,        
    pdfmenubar=true,        
    pdffitwindow=false,     
    pdfstartview={FitH},    
    pdftitle={},    
    pdfauthor={Leonid Monin},     
    pdfsubject={Subject},   
    pdfcreator={Leonid Monin},   
    pdfproducer={Producer}, 
    pdfkeywords={Linear series} {Spherical Varieties} {Newton Polyhedra}  {Newton Okounkov bodies} {Resultants} , 
    pdfnewwindow=true,      
    colorlinks=true,       
    linkcolor=blue,          
    citecolor=red,        
    filecolor=magenta,      
    urlcolor=cyan           
}

 \renewcommand{\epsilon}{\varepsilon}

\newcommand{\Z}{{\mathbb Z}}
 
\newcommand{\R}{{\mathbb R}}
\newcommand{\C}{{\mathbb C}}
\newcommand{\p}{{\mathbb P}}

\newcommand{\dd}{\operatorname{d}}
\newcommand{\ddd}{\operatorname{d\!}}
\newcommand{\Def}{\operatorname{def}}

\newcommand{\cA}{{\mathcal A}}

\newcommand{\cL}{{\mathcal L}}

\newcommand{\cE}{{\mathcal E}}
\newcommand{\cV}{{\mathcal V}}
\newcommand{\cW}{{\mathcal W}}
\newcommand{\cO}{{\mathcal O}}
\newcommand{\E}{{\bf E}}
\newcommand{\W}{{\bf W}}

\theoremstyle{remark}
\newtheorem{Definition}{Definition}

\def \p{\mathbb{P}}

\def \A{\mathbb{A}}

\def \/{\big/}

\def \SL{\textup{SL}}

\def \GL{\textup{GL}}

\def \dim{\textup{dim}}

\def \codim{\textup{codim}}
\def \ker{\textup{ker}}

\def \Spec{\textup{Spec}}

\usepackage{thmtools}

\theoremstyle{plain}
\newtheorem{Th}{Theorem}[section]

 \newtheorem{Proposition}{Proposition}[section]
 \newtheorem{Theorem}{Theorem}[section]
 \newtheorem{Corollary}{Corollary}[section]
  \newtheorem{Lemma}{Lemma}[section]
 \newtheorem*{NoNumberTheorem}{Theorem}
\theoremstyle{definition}

\newtheorem{Rem}[Th]{Remark}

\begin{document} 

\definecolor{adr}        {cmyk}{0.99,0.,0.,0.1}
\newcommand{\ed}[1]{\textbf{\textcolor{adr}{#1}}}

\title{Overdetermined  Systems of Equations on Toric, Spherical, and Other Algebraic Varieties}
\author{Leonid Monin}
\maketitle

\begin{abstract}
Let $E_1,\ldots,E_k$ be a collection of linear series on an irreducible algebraic variety $X$ over $\C$ which is not assumed to be complete or affine. That is, $E_i\subset H^0(X, \cL_i)$ is a finite dimensional subspace of the space of regular sections of line bundles $\cL_i$. Such a collection is called overdetermined  if the generic system
\[
s_1 = \ldots = s_k = 0,
\]
with $s_i\in E_i$ does not have any roots on $X$. In this paper we study consistent systems which are given by an overdetermined  collection of linear series.
Generalizing the notion of a resultant hypersurface we define a consistency variety $R\subset \prod_{i=1}^k E_i$ as the closure of the set of all systems which have at least one common root and study general properties of zero sets $Z_{\bf s}$ of a generic consistent system ${\bf s}\in R$. Then, in the case of equivariant linear series on spherical homogeneous spaces we provide a strategy for computing discrete invariants of such generic non-empty set $Z_{\bf s}$. For equivariant linear series on the torus $(\C^*)^n$ this strategy provides explicit calculations and generalizes the theory of Newton polyhedra. 
\end{abstract}

\tableofcontents

\section{Introduction}
Let $X$ be an irreducible algebraic variety over $\C$ and let $\cE=(E_1,\ldots,E_k)$ be a collection of base-point free linear series on $X$. That is, $E_i\subset H^0(X, \cL_i)$ is a finite dimensional subspace of the space of regular sections of globally generated line bundles $\cL_i$, such that there are no points $x\in X$ with $s(x)=0$ for any $s\in E_i$. 

A collection of linear series $\cE$ defines systems of equations on $X$ of the form
\begin{equation}\label{eq1}
s_1=\dots=s_k=0,
\end{equation}
where $s_i\in E_i$. A collection $\cE$ is called {\em overdetermined } if system~(\ref{eq1}) does not have any roots on $X$ for the generic choice of ${\bf s}=(s_1,\ldots,s_k)\in \E=E_1\times\ldots\times E_k$. If generic system~(\ref{eq1}) has a solution we will say that $\cE$ is generically consistent. Here, and everywhere in this paper, by saying that some property is satisfied by a generic point of an irreducible algebraic variety $Y$ we mean that there exists a Zariski closed subset $Z\subset Y$ such that for any $y\in Y\setminus Z$ this property is satisfied.

In this paper we study overdetermined  collections of linear series and their generic non-empty zero sets. For a collection $\cE=(E_1,\ldots,E_k)$ we define the consistency variety $R_\cE\subset \E$ to be the closure of the set of all systems $\bf s$ which have at least one common root on $X$. A consistency variety $R_\cE$ is an irreducible variety which is a generalization of the resultant hypersurface. We study $R_\cE$ in Section~\ref{consvar}. 

One of the main goals of this paper is to find discrete invariants of the zero set of a generic consistent system ${\bf s}\in R_\cE$. This goal is motivated by the study of families of complete intersections. A motivating example for us is a system of equations defining the singular locus of an algebraic hypersurface. Let $H$ be a hypersurface defined by a polynomial $f$ on $\C^n$. The singular locus of $H$ is given by the conditions $f=df=0$ which could be viewed as $n+1$ algebraic equations in $n$ variables. Since this system is overdetermined, the generic hypersurface is smooth. However, for a generic 1-parameter family of hypersurfaces $H_t$ defined by polynomials $f_t$ one would see singular hypersurfaces. In this setting, questions about the topology of the singular loci appearing in the generic families are translated to questions about the topology of generic non-empty zero sets of overdetermined  linear systems.

\smallskip
\noindent
{\bf Newton Polyhedra theory and generalisations.} There are lots of results on the zero sets of a system of equations defined by the generic member of a linear series. First such results are provided by the theory of Newton polyhedra which work with equivariant linear series on $(\C^*)^n$. 

The  Newton polyhedron $\Delta(f)\subset \R^n$ of a Laurent polynomial $f=\sum a_i x^{k_i}$ is the convex hull of vectors $k_i$ with $a_i\ne 0$. For a fixed polytope $\Delta$ let $E_\Delta$ be a finite dimensional vector space of Laurent polynomials $f$ such that $\Delta(f)\subset\Delta$. Newton polyhedra theory allows one to find discrete invariants of the zero set $Z_{\bf s}$ of the system (\ref{eq1}) where $s_i$ is a generic member of $E_{\Delta_i}$   in terms of combinatorics of the polytopes $\Delta_1,\ldots, \Delta_k$. An example of such a result is the celebrated Bernstein-Kouchnirenko-Khovanskii Theorem (see \cite{B75}). 
\begin{NoNumberTheorem}[BKK Theorem]
Let $s_1, \ldots, s_n$ be generic Laurent polynomials with $\Delta(s_i)=\Delta_i$. Then all solutions of the system $s_1 = \ldots = s_n = 0$ in $(\C^*)^n$  are non-degenerate and the number solutions is
$$
n! Vol(\Delta_1,\ldots,\Delta_n),
$$
where $Vol$ is the mixed volume.
\end{NoNumberTheorem}
For more examples of results of Newton polyhedra theory see \cite{Kh78, Kh88, DKh86}.

Newton polyhedra theory has generalizations to other classes of algebraic varieties such as spherical homogeneous spaces $G/H$ with a collection of $G$-invariant linear systems. The first result in this direction was a generalization of the BKK Theorem and was obtained by Brion and Kazarnovskii in \cite{Bri89, Kaz87}. For more results see for example \cite{Kir06, Kir07, KK16}. The role of the Newton polytope in these results is played by the  Newton-Okounkov polytope, which  is a polytope fibered over the moment polytope with string polytopes as fibers.

Even more generally, in \cite{KK12} and \cite{LM} Newton polyhedra theory was generalized to the theory of Newton-Okounkov bodies. For a  linear series $E$ on an irreducible algebraic variety $X$ one can associate a convex body $\Delta(E)$ called the Newton-Okounkov body in such a way that the number of roots of a generic system
$$
s_1=\ldots=s_n=0
$$
with $s_i\in E$ is equal to $n!Vol(\Delta(E))$.

All these results work for a generic system. That is, as before, in the space of systems $\E=E_1\times\ldots\times E_k$ there exists a Zariski closed subset $D$ such that for any system ${\bf s}\in \E \setminus D$ discrete invariants of the zero set $Z_{\bf s}$ are the same and can be computed combinatorially. In particular, for overdetermined  systems all the answers provided by these results are trivial.

\smallskip
\noindent
{\bf Structure of the paper, formulation of the results and related previous work.} Studying a generic consistent system given by an overdetermined  collection of linear series is a particular case of studying generic among non-generic systems. For generically consistent collections of linear series, the study of generic among non-generic systems is usually hard. Even in the case of the BKK Theorem such results are quite technical and recent (see \cite{Pin}). However, if a collection of linear systems is overdetermined, there are lots of cases where questions about the topology of generic non-empty zero sets could be answered rather explicitly in terms of combinatorics. First results in this direction were already obtained in \cite{GKZ}. It was shown there that for collections of finite sets $A_1,\ldots,A_{n+1}\subset \Z^n$ which generate the lattice, the generic consistent system of Laurent polynomials with supports in $A_i'$s has a unique root. This result was later generalised by different authors for more general undetermined systems of Laurent polynomials (see \cite{esterov2007, esterov2010, D'AS}). These results were further generalized in \cite{Mon17} and described in Section~\ref{torus} of this paper.

In Section~\ref{consvar}, for an overdetermined  collection of linear series, we define the consistency variety $R_\cE\subset \E$ which is the closure of the set of all consistent systems and prove that it is irreducible. Basic geometric properties of $R_\cE$ are studied in Theorems~\ref{codim} and \ref{indep}. In Subsection~\ref{resultant} we consider the case when $codim(R_\cE)=1$ and define the resultant polynomial of a collection $\cE$. The resultant of a collection of linear series is a generalization of the $\cL$-resultant defined in \cite{GKZ}. We prove that all the basic properties of the $\cL$-resultant are also satisfied by the resultant of a collection of linear series. A related version of resultant was also studied in \cite{Bus} in the case when $X$ is projective with an open subset $U$ parametrized by an open subset of $\A^n$. In the case when conditions of \cite{Bus} are satisfied, our resultant is a non-trivial degree of the one defined in \cite{Bus} (see Definition~\ref{resdef} and a footnote after it for details).

Section~\ref{zerogen} is devoted to studying the generic non-empty zero set of an overdetermined  collection of linear systems $\cE$ on an irreducible variety $X$. One of the main results of this section is Theorem~\ref{maingen} which expresses a generic non-empty zero set of system~(\ref{eq1}), defined by $\cE$, as the generic zero set of another linear series which is generically consistent. This allows one to use classical results described in the previous subsection to find the topology of the generic non-empty zero set in a number of examples.

In Section~\ref{hom} we study $G$-equivariant linear series on homogeneous $G$-spaces. Spherical homogeneous spaces are of special interest to us. We apply Theorem~\ref{maingen} to obtain Theorem~\ref{mainsph} which provides a strategy for computing discrete invariants of a generic non-empty zero set of an overdetermined  linear series on a spherical homogeneous space. An example of an application of Theorem~\ref{mainsph} is given in Subsection~\ref{ex}.

In Section~\ref{torus} we study overdetermined  linear systems on the torus $(\C^*)^n$. In this case the strategy provided by Theorem~\ref{mainsph} can be made absolutely explicit. In particular, one can explicitly express invariants of generic non-empty zero sets in terms of combinatorics of Newton polyhedra. Theorem~\ref{toric} is an explicit version of Theorem~\ref{mainsph} in this case. Theorem~\ref{BKK}, which generalizes the BKK-theorem, is an example of how one can use Theorem~\ref{toric} together with results of Newton polyhedra theory to compute discrete invariants of generic non-empty zero sets. The results of Section~\ref{torus} appeared previously in \cite{Mon17} and are included in this paper for completeness of exposition.

\noindent
{\bf Acknowledgments.} The author would like to thank Askold Khovanskii for his enthusiasm and support during the work and the  anonymous referee for their comments which helped to improve this paper.

\section{Consistency variety and resultant of a collection of linear series on a variety}\label{consvar}
In this section we define the consistency variety of a collection of linear series and describe its main properties. 

\subsection{Background} 
Let $X$ be an irreducible complex algebraic variety. Let $\cL_1,\ldots \cL_k$ be globally generated line bundles on $X$. For $i = 1,\ldots, k$, let $E_k \subset H^0(X,\cL_i)$ be finite-dimensional, base-point free linear series. Let $\bf E$ denote the $k$-fold product $E_1 \times \cdots\times E_k$.

\begin{Definition}
The {\it incidence variety} $\widetilde R_{\bf E}\subset X\times \E$ is defined as:
$$
\widetilde R_\E=\{(p,(s_1,\ldots,s_k))\in X\times \E \ |\ f_1(p)=\ldots=f_k(p)=0 \}.
$$
\end{Definition}

Let $\pi_1:X\times \E \to X$, $\pi_2:X\times \E \to  \E$  be natural projections to the first and the second factors of the product. 

\begin{Definition}
The {\it consistency variety} $R_\E\subset \Omega_A$ is the closure of the image of $\widetilde R_L$ under the projection $\pi_2$.
\end{Definition}

\begin{Theorem}
The  incidence variety $\widetilde R_\E\subset X\times L$ and the  consistency variety $R_\E$ are  irreducible algebraic varieties. 
\end{Theorem}

\begin{proof}
Since $E_1,\ldots E_k$ are base-point free, the preimage $\pi_1^{-1}(p)\subset \widetilde R_\E$ of any point $p\in X$ is defined by $k$  independent linear equations on elements of $\E$. Therefore, the projection $\pi_1$ restricted to $\widetilde R_\E$:
$$
\pi_1:\widetilde R_\E \to X
$$
forms a vector bundle of rank $\dim(L) -k$ and in particular is irreducible.  

The set of consistent systems $R_\E=\pi_2(\widetilde R_\E)$ is the image of an irreducible algebraic variety $\widetilde R_\E$ under the algebraic map $\pi_2$, so it is irreducible constructible set. Hence it's closure is an irreducible algebraic variety.
\end{proof}

For two linear systems $E_i$ and $E_j$, let their product $E_iE_j$ be a vector subspace of $H^0(X,\cL_i\otimes\cL_j)$ generated by all the elements of the form $f\otimes g$ with $f\in E_i$ and $g\in E_j$. For any $J\subset \{1,\ldots,k\}$, by $E_J$ we will denote the product $\prod_{j\in J} E_j$.

To a base-point free linear system $E$, one can associate a morphism $\Phi_E:X\to \p(E^*)$ called a {\em Kodaira map}. It is defined as follows: for a point $x\in X$ its image $\Phi_E(x) \in \p(E^*)$ is the hyperplane  $E_x\in E$ consisting of all the sections $g\in E$ which vanish at $x$. We will denote by $Y_E$ the image  of the Kodaira map $\Phi_E$ and by $\tau_E$ the dimension of $Y_E$. For a collection of linear series $E_1,\ldots,E_k$ and $J\subset \{1,\ldots,k\}$ we will write $\Phi_J, Y_J$ and $\tau_J$ for $\Phi_{E_J}, Y_{E_J}$ and $\tau_{E_J}$ respectively.

\begin{Definition}\label{defK}
For a collection of linear series $E_1\ldots,E_k$, the \emph{defect} of a subcollection $J\subset \{1,\ldots,k\}$ is defined as 
$$
\Def(E_J)= \tau_J - |J|.
$$
\end{Definition}

The following theorem of Kaveh and Khovanskii gives a condition on a collection of linear series to be generically consistent in terms of defects.

\begin{Theorem}[\cite{KK16} Theorems 2.14 and 2.19]\label{consistentiff}
The generic system of equations $s_1=\ldots=s_k=0$ with $f_i\in E_i$ is consistent if and only if $\Def(E_J)\geq 0$ for any $J\subset \{1,\ldots,k\}$.
\end{Theorem}

In other words, Theorem~\ref{consistentiff} states that the codimension of the consistency variety is equal to 0 in $\E$ if and only if $\Def(E_J)\geq 0$ for any $J\subset \{1,\ldots,k\}$. In Subsection~\ref{prop} we will generalize this result by finding the codimension of $R_E$ in terms of defects of subcollections.

\subsection{The defect of vector subspaces and essential subcollections}  In this subsection we will introduce a combinatorial version of the defect and relate it to the one given in Definition~\ref{defK}. Let $k$ be any field and $\cV= (V_1,\ldots, V_k)$ be a collection of vector subspaces of a vector space $W\cong k^n$ (in this paper we will always work with $k=\C$ or $\R$). 
For $J \subset \{1,\ldots,k\}$ let $V_J$ be the sum $\sum_{j \in J}V_j$ and $\pi_J:W \to W/V_J$ be the natural 
projection. 

\begin{Definition}\label{defc} For a collection of vector subspaces $\cV=(V_1,\ldots, V_k)$ of $W$ we define\\
 $i)$ the \emph{defect} of a  subcollection $J\subset \{1,\ldots,k\}$ by $\Def(J)=\dim(V_J) -  |J|$;\\
 $ii)$ the \emph{minimal defect} $\dd(\cV)$ of a collection $V$ to be the minimal defect of $J\subset \{1,\ldots,k\}$. \\
 $iii)$ an \emph{essential subcollection} to be a subcollection $J$ so that $\Def(J)=d(\cV)$ and $\Def(I)>\Def(J)$ for any proper subset $I$ of $J$.
\end{Definition}

The essential subcollections have proved to be useful in studying systems of equations which are generically inconsistent and their resultants. The above definition is related to the definition of an essential subcollection given in \cite{St94}: the two definitions coincide if $\dd(\cV)=-1$, but are different in general. The following theorem provides combinatorial tools to work with collection of vector subspaces. 

\begin{Theorem}[\cite{Mon17} Section 3]\label{unique}
Let $\cV=(V_1,\ldots, V_k)$ be a collection of vector subspaces of  $W$ with $\dd(\cV)\leq0$, then:\\
i) an essential subcollection  exists and is unique.\\
ii) if $J$ is the unique essential subcollection of $\cV$ and $J^c$ isv the complement subcollection, then $\dd(\pi_J(J^c))=0$.\\ 
iii) if $J$ is the essential subcollection there exists a subcollection $I\subset J$ of size $dim(V_J)$ with $\dd(I)=0$.
\end{Theorem}

Parts $i)$ and $ii)$ of Theorem~\ref{unique} are still true in the case $\dd(\cV)>0$, the unique essential subcollection in this case is the empty subcollection. A subcollection $I$ from part $iii)$ of Theorem~\ref{unique} is almost never unique.

To relate the combinatorial version of defect to the geometric version defined in Definition~\ref{defK} we  introduce a collection of distributions (and codistributions) on $X$ related to a collection of linear systems $E_1,\ldots,E_k$. Let  $X$ and $E_1,\ldots,E_k$ be as before, denote further  by $X^{sing}$ the singular locus of $X$, by $Y_J^{sing}$ the singular locus of $Y_J=\Phi_J(X)$.

Let also $\Sigma_J^c\subset X \setminus \left(X^{sing}\cup \Phi_J^{-1}(Y_J^{sing}) \right)$ be the set of all critical points of $\Phi_{J}$ viewed as a map to $Y_J\setminus Y_J^{sing}$. Finally, let $B_J=X^{sing}\cup\Phi_J^{-1}(Y_J^{sing})\cup\Sigma_J^c$ and let $U\subset X$ be a Zariski open subset defined by:
$$
U=X \,\big\backslash \bigcup_{J\subset\{1,\ldots,k\}}B_J.
$$
So we get that $U$ is a smooth algebraic variety and for any $J\subset \{1,\ldots,k\}$ the restriction $\Phi_J: U\to (Y_J\setminus Y_J^{sing})$ is a submersion to a smooth locus of $Y_J$. In particular, the rank of the differential $\ddd \Phi_J$ of $\Phi_J$ is constant on $U$ and is equal to $\tau_J$ for any $J$.

In the case of equivariant linear series on homogeneous space $G\/H$ (in particular, in the classical case of linear series on $(\C^*)^n$) subvarieties $Y_J^{sing}$ and $\Sigma_J^c$ are empty. In general, this is not true: 
\begin{itemize}
    \item For an example of $Y_J^{sing}$ being non empty consider $X=\A^1$ with a coordinate $x$ and a linear series $E$ on $X$ spanned by regular functions $1,x^2,x^3$. Then the image of $X$ under the Kodaira map:
    $$
    \Phi: X \to \p(E^*)\simeq \p^2=[z_0:z_1:z_2]
    $$
    is contained in the affine chart $z_0\ne 0$ and is defined in it by the unique equation $z_1^3-z_2^2=0.$ In particular, $Y^{sing}=[1:0:0]$ and is non-empty.
    \item For an example where $\Sigma_J^c\ne \varnothing$, let $X=\A^2$ with coordinates $x,y$, and $E= span(1,x^2+y^2)$. Then the Kodaira map $\Phi:X\to \p(E^*)=[z_0:z_1]$ is given by:
    $$
    \Phi(x,y) = [1:x^2+y^2].
    $$
    The image $\Phi(X)$ is an affine chart $z_0\ne 0$ and is in particular smooth. Nevertheless, the differential $\ddd\Phi = 2x\ddd x+2y\ddd y$ vanishes at the point $(0,0)$ and has rank 1 everywhere else on $X$. So the critical locus $\Sigma^c \subset  X \setminus \left(X^{sing}\cup \Phi^{-1}(Y^{sing}) \right)$ in this case is a point $(0,0)\in X$.
\end{itemize}

\begin{Definition} 
Let $a \in U$ and $\widetilde{F}_J(a)$ be the subspace of the tangent space $T_aU$ defined by the linear equations $dg_a = 0$ for all $g \in E_J$. Let also $\widetilde F_J^\vee(a) \subset T^*_aU$ be the annihilator of $\widetilde{F}_J(a)$. Then\\
(1) $\widetilde{F}_J$ is an $(n - \tau_J)$-dimensional distribution on the Zariski open set $U\subset X$ defined by the collection of subspaces $\widetilde{F}_J(a)$.\\
(2)  $\widetilde F_J^\vee$ is a $\tau_J$-dimensional codistribution on the Zariski open set $U\subset X$ defined by the collection of subspaces $\widetilde{F}_J^\vee(a)$.
\end{Definition}

\begin{Lemma} \label{integrable}
The distribution $\widetilde{F}_J$ and codistribution $\widetilde{F}^\vee_J$ on $U$ are well defined, i.e. $\dim(\widetilde{F}_J(a))=n-\tau_J$ and $\dim(\widetilde{F}_J(a))=\tau_J$ for any $a\in U$. Moreover, $\widetilde{F}_J$ is completely integrable, its leaves are  connected components of the fibers the Kodaira map $\Phi_J: U \to Y_J$.
\end{Lemma}
\begin{proof}
For the first part of the lemma it is enough to show that $\dim(\widetilde{F}_J(a))=n-\tau_J$ for any $a\in U$. Let $s_1,\ldots,s_r \in E_J$ be a basis, then in coordinates the Kodaira map is given by:
$$
\Phi_J(x) = [s_1(x):\ldots:s_r(x)].
$$
Therefore, the vector space $\widetilde{F}_J(a)$ is the kernel of differential $\dd(\Phi_J)$ of the Kodaira map. Since $\dd(\Phi_J)$ is of constant rank $\tau_J$ on $U$, 
$$
\dim(\widetilde{F}_J(a))= \dim\left( \ker \dd(\Phi_J)\big|_a\right) = n-\tau_J.
$$

The second part of the lemma is an immediate corollary the Implicit Function Theorem. Indeed, by Implicit Function Theorem the the fibers of $\Phi_J: U \to Y_J$ are smooth subvarieties of $X$ of dimension $n-\tau_J$, tangent to $\ker(\dd(\Phi_J))=\widetilde{F}_J$ on $U$. 
\end{proof}

Fibers of the Kodaira maps $\Phi_J$ can be described in terms of systems of equations defined by collection of linear series $E_1,\ldots, E_k$.
\begin{Lemma}[\cite{KK16} Lemma 2.11]\label{fiber}
 For $a,b\in X$ we have $\Phi_J(a)=\Phi_J(b)$ if and only if for every $i\in J$ the sets $\{g_i\in E_i| g_i(a)=0\}$ and  $\{g_i\in E_i| g_i(b)=0\}$ coincide.
\end{Lemma}

\begin{Corollary}\label{intersectsum}
Let $U\subset X$ be as before, then for any $a\in U$ one has:
$$
F_{E_J}(a)=\bigcap_{i\in J}F_i(a), \quad F_{E_J}^\vee(a)=\sum_{i\in J}F_i^\vee(a).
$$
\end{Corollary}
\begin{proof}

Indeed, by Lemma~\ref{fiber} the following diagram is commutative

\begin{center}
   \begin{tikzcd}
\prod_{i\in J} \p(E_i^*)  \arrow[rr, hook, "s"] & &
\p\left(\displaystyle\bigotimes_{i\in J} E_i^*\right)& & 
\p(E_J^*) \arrow[ll,hook', "i"']\\
& && & \\
& &X\arrow[lluu, "\prod_{i\in J} \Phi_i"]  \arrow[rruu, " \Phi_J"']& &
\end{tikzcd}
\end{center}
where $s$ is a Segre embedding and $i$ is the projectivisation of a dual map to the natural projection:
$$
\bigotimes_{i\in J} E_i \twoheadrightarrow E_J.
$$
By the proof of Lemma~\ref{integrable}, $\widetilde F_{E_J}(a)$ is the kernel of the differential of the Kodaira map $\ddd\Phi_J$ restricted to $T_aU$. But since $i\circ\Phi_J =s\circ\prod_{i\in J} \Phi_i$ (and both $s$ and $i$ are embeddings) we have:
$$
\ker \dd(\Phi_J)\big|_a = \ker \ddd\left.\left(\prod_{i\in J} \Phi_i\right)\right|_a =  \displaystyle \bigcap_{i\in J}\ker \dd(\Phi_i)\big|_a,
$$
and hence $F_{E_J}(a)=\displaystyle\bigcap_{i\in J}F_i(a)$. The second statements follows immediately by dualization. 
\end{proof}

The following Proposition~\ref{defisdef} relates two definitions of the defect and is the main result of this subsection. Proposition~\ref{defisdef} will allow us to apply combinatorial results of Theorem~\ref{unique}  to the geometric version of defect.

\begin{Proposition}\label{defisdef}
Let $U\subset X$ be as before, then the defect $\Def(E_J)$  of linear system $E_J$ (as in Definition~\ref{defK}) is equal to the defect of a collection of vector subspaces $(F_{E_i}^\vee(a))_{i\in J}$ of $T^*U_a$ (as in Definition~\ref{defc}) for any $a\in U$.
\end{Proposition}

\begin{proof}
By construction, $F_{E_J}^\vee$ is $\tau_J$-dimensional codistribution on $U$, so $\dim(F_{E_J}^\vee(a))=\tau_J$ for any $a\in U$. Therefore, by Corollary~\ref{intersectsum} we have

\begin{equation*}
\begin{split}
\Def(E_J)= \tau_J -|J| = \dim(F_{E_J}^\vee(a)) -|J|= \dim(F_{E_J}^\vee(a)) -|J|\\
=\dim\left(\sum_{i\in J}F_i^\vee(a)\right)-|J|=\Def (F_{E_i}^\vee(a))_{i\in J}. \qedhere
\end{split}
\end{equation*}
\end{proof}

\subsection{Properties of the consistency variety}\label{prop} In this subsection we will investigate basic properties of the consistency variety. One of the main results of this section is the following theorem which computes the codimension of the consistency variety in terms of defects. 

\begin{Theorem}\label{codim}
Let $\cE=(E_1,\ldots E_k)$ be a collection of base-point free linear systems on a quasi projective irreducible variety $X$. Then the codimension of the consistency variety $R_\E$ is equal to $-\dd(\cE)$ where $\dd(\cE)$ is the minimal possible defect $\Def(E_J)$ for $J\subset \{1,\ldots,k\}$. 
\end{Theorem}

We say that a collection of linear series $\cE=(E_1,\ldots,E_k)$ on $X$ is {\em injective} if linear series from $\cE$ separate points of $X$. In other words, $\cE$ is  injective if the product of Kodaira maps $\prod_{i=1}^k \Phi_{E_i}$ is injective on $X$. Note that by Lemma~\ref{fiber} the product of Kodaira maps $\prod_{i=1}^k \Phi_{E_i}$ is injective   if and only if  the Kodaira map $\Phi_E$ for $E=\prod_{i=1}^k E_i$ is such. Therefore, equivalently $\cE$ is injective if  the Kodaira map $\Phi_E$ is injective.

Any collection of linear series $\cE$ on $X$ could be reduced to an injective collection $\widetilde \cE $ such that zero sets of $\cE$ and $\widetilde \cE$ are related in an easy way.  In order to do so, let us describe the zero set $Z_{\bf s}$ of a system of equations $s_1=\ldots=s_k=0$ with $s_i\in E_i$ in terms of Kodaira maps $\Phi_{E_i}$.

For the product of projective spaces $\p_{\E}=\p(E_1^*)\times\ldots\times \p(E_k^*)$ let $p_i:\p_{\E}\to \p(E_i^*)$ be the natural projection on the $i$-th factor. Each function $s_i\in E_i$ defines a hyperplane $H_{s_i}$ on $\p(E_i^*)$, with slight abuse of notation let us denote its preimage under $p_i$ by the same letter. Let $\Phi_\E: X\to \p_\E$ be the product of Kodaira maps and $Y_\E = \Phi_\E(X)$ be its image. In this notation the zero set $Z_{\bf s}$ is given by 
$$
Z_{\bf s}=\Phi_\E^{-1}\left(Y_\E \cap \bigcap_{i=1}^{k}H_{s_i}\right).
$$ 

Therefore for any collection of linear series $\cE=(E_1,\ldots,E_k)$ on $X$ one can associate an injective collection $\widetilde \cE$ on $Y_\E$, where   $\widetilde \cE$  is the restriction of $E_1,\ldots,E_k$ to  $Y_\E$ such that
$$
Z_{\bf s} = \Phi_\E^{-1}(Z_{\bf \widetilde{s}}).
$$

\begin{Proposition}\label{solv}
Let $X$ be an irreducible algebraic variety over $\C$ and $\cE=(E_1,\ldots,E_k)$ be an overdetermined  collection of linear systems on $X$, let $J$  be the essential subcollection of $\cE$. Let also $X_a$ be a fiber of $\Phi_J$ passing through a point $a\in X$. Then for generic $a\in X$, the restriction of the collection $\cE_{J^c}=(E_i)_{i\notin J}$ on $X_a$ is generically consistent.
\end{Proposition}

\begin{proof}
Here one can take any $a\in U$, with $U=X\setminus \bigcup_{I\subset \{1,\ldots,k\}} B_I$ as before. By Theorem~\ref{consistentiff} it is enough to show that the minimal defect of the restriction of collection of linear series $\cE_{J^c}$ on $X_a$ is nonnegative. Let $J^c = \{i_1,\ldots,i_s\}$. By Proposition~\ref{defisdef} the minimal defect of the collection of linear series $\cE_{J^c}$, can be computed as the minimal defect of a collection of vector subspaces 
$$
\widetilde F_{i_1}^\vee(a) ,\ldots, \widetilde F_{i_s}^\vee(a) \subset T^*_aX_a,
$$
where $\widetilde F_{i_j}^\vee$ is the codistribution of the restriction of $E_{i_j}$ to $X_a$. The codistributions $\widetilde F_{i_1}^\vee ,\ldots, \widetilde F_{i_s}^\vee$ are given by $\pi_J(F^\vee_{i_1}),\ldots,\pi_J(F^\vee_{i_1})$, where 
$$
\pi_J:T^*U\to T^*(X_a\cap U) \cong T^*U/F^\vee_J,
$$
is the natural projection. By the second part of Theorem~\ref{unique} we know that $d(\pi_J(F^\vee_{i_1}),\ldots,\pi_J(F^\vee_{i_1}))=0$, so the theorem is proved.
\end{proof}

\begin{proof}[Proof of Theorem~\ref{codim}]
Without loss of generality we can assume that $\cE$ is injective. Indeed,  $\widetilde \cE$ is consistent in codimension $r$ if and only if $\cE$ is consistent in codimension $r$. 

By Proposition~\ref{solv} we can assume that the essential subcollection of $E_1,\ldots,E_k$ is the collection itself. We will call such collections of linear series {\em essential}. For an essential collection, by Theorem~\ref{unique} there exists a consistent subcollection of size $\tau_\E$, assume it is equal to $I= \{1,\ldots,r\}$. Then the dimension $\tau_\E$ is equal to the dimension $\tau_I$. Therefore the generic solution linear conditions from $I$ on $\Phi_I(X)$ is a union of finitely many points. 

Since linear systems $E_i$'s are base-point free, the condition on the sections from $E_{r+1},\ldots,E_k$ to vanish at any of these points is union of $k-r=-\Def(\E)$ clearly independent linear conditions, which finishes the proof of the theorem in this case. 
\end{proof}

We finish this section with another corollary of Proposition~\ref{solv} which reduces the study of resultant subvarieties to the study of resultant subvarieties of essential collections of linear systems.

\begin{Theorem}\label{indep}
Let $X$ be a complex irreducible quasi-projective algebraic variety and $\cE=(E_1,\ldots,E_k)$ be overdetermined  collection of linear systems on $X$ with the essential subcollection  $J$. Then the consistency variety $R_\cE$ does not depend on $E_i$ with $i\notin J$. In other words:
$$
R_\cE=p^{-1}(R_{\cE_J}), \text{ where } p:\E\to \E_J =\prod_{i\in J}E_i
$$
is the natural projection.
\end{Theorem}

\begin{proof}
By Proposition~\ref{solv} there exists a Zariski open subset $W$ of $R_{\cE_J}$ such that for any ${\bf s} \in W$ the collection of linear systems $J^c$ restricted to a zero set of ${\bf s}$ is generically consistent. Therefore if $V\subset R_\cE$ is a set of consistent systems $\bf s$ such that $p({\bf s})\in U$ one has $\codim(V)=\codim(R_\cE)=-d(\cE)$. Since $R_\E$ is an irreducible variety it coincides with the closure of $V$,  which finishes the proof.
\end{proof}

\subsection{Resultant of a collection of linear series on a variety}\label{resultant} In this subsection we will define resultant of a collection of linear series and translate results of previous subsection to the language of resultants.  The notion of resultant defined here is a generalization of $\cL$-resultant defined in \cite{GKZ}, and most of the results are analogous to the results on $\cL$-resultants in \cite{GKZ}.
 
Let $\cE=(E_1,\ldots,E_{n+1})$ be a collection of linear systems on an irreducible variety $X$ of dimension $n$. Assume also, that the codimension of the consistency variety $R_\cE$ is equal $1$.
\begin{Lemma}
Let $X, \cE$ be as before, then there exists an Zariski open subset $U\subset R_\cE$ so that for any ${\bf s} = (s_1,\ldots,s_{n+1})\in U$, the zero set $Z_{\bf s}$ of the system $s_1=\ldots=s_{n+1}=0$ on $X$ is finite. Moreover one can choose $U$ such that the cardinality of $Z_{\bf s}$ is the same for any ${\bf s} \in U$.
\end{Lemma}
 \begin{proof}
Let $\pi_1, \pi_2$ be restrictions of two natural projections from $X\times\E$ to $X,\E$ respectively to a incidence variety $\widetilde R_\cE$. For a system ${\bf s}\in R_\cE$ the zero set $Z_{\bf s}$ is given by $\pi_1(\pi_2^{-1}({\bf s}))$, in particular if $\pi_2^{-1}({\bf s})$ is finite of cardinality $k$ such is $Z_{\bf s}$. Easy dimension counting shows that $\dim \widetilde R_\cE =\dim R_\cE$, so for the generic ${\bf s}\in R_\cE$ the preimage $\pi_2^{-1}({\bf s})$ is finite, and of the fixed cardinality.
 \end{proof}
 
Note that the conditions that $\codim R_\cE = 1$ and generic non-empty zero set is finite forces the number of linear series to be $n+1$.
 
\begin{Definition}\label{resdef}
Let $X, \cE$ be as before, then the resultant $Res_\cE$ is a polynomial which defines the hypersurface $R_\cE$ with multiplicity equal to the number of points in the generic non-empty zero set $Z_{\bf s}$ \footnote{In some places (\cite{Bus, St94}) the resultant is defined as unique up to constant irreducible polynomial defining $R_\cE$, but the definition provided here seems more natural. See  \cite{esterov2007, esterov2010, D'AS} for details.}. Since $R_E$ is irreducible such polynomial is well  defined up to multiplicative constant. For $(s_1,\ldots,s_{n+1}) \in \E$ by $Res_\cE(s_1,\ldots,s_{n+1})$  we will denote the value of resultant on the tuple $s_1,\ldots,s_{n+1}$.
\end{Definition}

The next theorem is an immediate corollary of Theorems~\ref{codim} and~\ref{indep}. This theorem was proved by Sturmfels  in \cite{St94}  for equivariant linear series on an algebraic torus. In that setting resultant of a collection of linear series on a variety is usually called {\em sparse resultant}.
 
\begin{Theorem}
The consistency variety $R_\cE$ of a collection of linear systems $\cE=(E_1,\ldots,E_{n+1})$ has codimension 1 if and only if  $\dd(E_1,\ldots,E_{n+1})=-1$. Moreover, if $J$ is essential subcollection of $\cE$, then the resultant $Res_\cE$ depends only on equations from $E_i$ with $i\in J$.
\end{Theorem}

For a collection $E_1,\ldots,E_n$ of linear series on an irreducible variety $X$ the number of roots of a system $s_1=\ldots=s_n=0$ is constant for the generic $s_i\in E_i$. The generic number of roots of a system $s_1=\ldots=s_n=0$ is called {\it the intersection index} of $E_1,\ldots,E_n$ and is denoted by $[E_1,\ldots,E_n]$. 

\begin{Theorem}
The resultant $Res_\cE$ is a quasihomogeneous polynomial with degree in the $i$-th entry equal to the intersection index $[E_1,\ldots,\hat E_i,\ldots,E_{n+1}]$. In particular, if $J$ is essential subcollection of $E$ and $i\notin J$ the degree in the $i$-th entry is 0.
\end{Theorem}
  
\begin{proof}
The resultant $Res_\cE$ is a homogeneous polynomial in each group of variables, since a system $s_1=\ldots=s_n=0$ has a root on $X$, if and only if  a system $\lambda_1 s_1=\ldots=\lambda_n s_n=0$ has a root for any $\lambda_i \in \C^*$. To find the degree of $Res_\cE$ in the $i$-th entry consider 
$$
Res_\cE(s_1,\ldots, s_i+\lambda s_i', \ldots, s_{n+1})
$$
as a polynomial of $\lambda$ for the fixed generic choice of $s_1,\ldots, s_i, s_i', \ldots, s_{n+1}$. It is easy to see that the number of roots of $Res_\cE(s_1,\ldots, s_i+\lambda s_i', \ldots, s_{n+1})$ counting with multiplicities is equal to the number of common roots of $s_1=\ldots=\hat s_i=\ldots =s_{n+1}=0$, so the degree of $Res_\cE$ in the $i$-th entry is $[E_1,\ldots,\hat E_i,\ldots,E_{n+1}]$.
\end{proof}


 \section{Generic non-empty zero set and reduction theorem}\label{zerogen}
 In this section we first study generic non-empty zero sets. In particular, we show in Theorem~\ref{maingen} that a generic non-empty zero set given by an overdetermined  collection of linear series can be also defined as a generic zero set of generically consistent collection. Then we define a notion of equivalence of two collections of linear systems. Informally speaking, two collections of linear systems are equivalent if they have the same generic nonempty zero sets. We show that every generically inconsistent collection of linear series is equivalent to a collection of minimal defect~$-1$.  
 
\subsection{Zero sets of essential collection of linear systems} First, we will study {\em essential} collections of linear series i.e. collections $\cE=(E_1,\ldots, E_k)$ such that $\Def (J)>\dd(\cE)$ for any $J\varsubsetneq \{1,\ldots,k\}$.

Let $Y\subset \p_\E=\p(E_{1}^*)\times\ldots\times\p(E_k^*)$ be an irreducible variety of dimension $d$. For a subset $J\subset \{1,\ldots,k\}$ denote by  $\p_J$ the product $\prod_{i\in J}\p(E_i)$ and by $\pi_J$ the natural projection $\p({\bf E^*}) \to \p_J$. With slight abuse of notation let us denote the restriction of this projection on  $Y$ also by $\pi_J$ and by $Y_J$ the image of the restricted map. Assume also, that $\cE=(E_1,\ldots, E_k)$ is an essential collection  of linear systems on $Y$ with a consistent subcollection $J$ of size $d$ which exists by Theorem~\ref{unique} and Proposition~\ref{defisdef}.

\begin{Lemma}\label{essdisjoint}
In the situation above for the generic pair of points $x_1,x_2\in Y_J$ the sets $F_{x_i}=\pi_{J^c}(\pi_J^{-1}(x_i))$, for $i\in \{1,2\}$ are disjoint.
\end{Lemma}
\begin{proof}
The condition on sets $F_{x_1}$ and $F_{x_2}$ to be disjoint is open in the space of pairs, so since $Y$ is irreducible it is enough to show that there exists at least one pair $x_1,x_2$ with $F_{x_1}\cap F_{x_2} =\varnothing$. 

Assume otherwise, then for a given point $x_0\in Y_J$ there exists an preimage $y_0\in \pi_J^{-1}(x_0)$ and an open set $U\in Y_J$ such that for any $x\in U$ there exists $y\in  \pi_J^{-1}(x)$ with $\pi_{J^c}(y)=\pi_{J^c}(y_0)$. So there exists a section $s: U \to Y$ of $\pi_J$ defined by $s(x)=y$ with a property that $\pi_{J^c}\circ s$ is a constant map on $U$. But since $\pi_J$ is a finite morphism, the image $s(U)$ is an Zariski open in $Y$, and, therefore $\pi_{J^c}$ is constant on $Y$, which contradicts the essentiallity of  $E_1,\ldots, E_k$.
\end{proof}

The main result of this subsection is the following proposition.
 \begin{Proposition}\label{esszero}
 Let  $E_1,\ldots,E_k$ be an essential collection  of linear systems on a quasi-projective irreducible variety $X$. Then for the generic point ${\bf s}\in R_\E$, the zero set $Z_{\bf s}$ of a system given by ${\bf s}$ is a single fiber of the Kodaira map $\Phi_E$.
 \end{Proposition}

 \begin{proof}
As in subsection~\ref{prop} we can assume that $\cE$ is injective by replacing $X$ with $Y_\E= \Phi_\E(X)\subset \p_\E=\prod_i E_i^*$. Note that the collection $\cE$ restricted to $Y_\E$ is still essential.

Therefore, it is enough to show that for an irreducible variety $Y_\E\subset \p_\E$ so that $\cE$ is an essential collection on $Y_\E$, and for the generic choice of ${\bf s} \in R_\cE$ the intersection $Y \cap H_{s_1}\cap\ldots\cap H_{s_k}$ is a point. 

 Let $J\subset \{1,\ldots,k\}$ be consistent subcollection of size $\tau_\E$, $J^c$ be its complement and let $\pi_J, \pi_{J^c}$ be two natural projections restricted to $Y_\E$:

\begin{center}
\begin{tikzcd}
Y_\E \arrow[r, "p_{J^c}"] \arrow[dd, "p_{J}"]& \p_{J^c}=\prod_{i\notin J} \p(E_i^*)\\
&& \\
\p_J=\prod_{i\in J} \p(E_i^*)& 

\end{tikzcd}    
\end{center}
The generic intersection $Y_J \cap \left(\bigcap_{i\in J}H_{s_i} \right)$ is nonempty and finite, and hence of the same cardinality. It is enough to show that for for generic choice of $s_i$'s with $i \in J$ there are no two points $x,y$ in $Y_J \cap \left(\bigcap_{i\in J}H_{s_i} \right)$  with $p_{J^c}(p_{J}^{-1}(x))\cap p_{J^c}(p_{J}^{-1}(y))\ne \varnothing$. Indeed, in such a case any two points in the finite intersection 
$$
Y_\E \cap \bigcap_{i\in J}H_{s_i} =\pi_J^{-1}\left(Y_J \cap \bigcap_{i\in J}H_{s_i} \right)
$$
 would be separated by generic hyperplanes $H_{s_i}$'s with $i\notin  J$. 

Let the cardinality of the generic intersection $Y_J \cap \left(\bigcap_{i\in J}H_{s_i} \right)$ be equal to $r$, we will show that for the generic $r$-tuple of points $x_1,\ldots,x_r$ the sets $F_{x_i}=\pi_{J^c}(\pi_J^{-1}(x_i))$, for $i=1,\ldots,r$ are  mutually disjoint. Since this condition is open in the space of tuples $x_1,\ldots, x_r$ and $Y_J$ is irreducible it is enough to show that there exist at least one tuple with such property. 

Assume otherwise, that for any tuple $x_1,\ldots,x_r$ the sets $F_{x_i}$, for $i=1,\ldots,r$ are not mutually disjoint. This is only possible if for any pair of point $x_1, x_2$ the sets $F_{x_1}, F_{x_2}$ are not disjoint, but this contradicts Lemma~\ref{essdisjoint} since $Y_\E$ satisfy its conditions.
 \end{proof}

\begin{Rem}
The condition of collection $\cE$ to be essential is sufficient for the statements of Lemma~\ref{essdisjoint} and Proposition~\ref{esszero} but not necessary.  Indeed, consider a pair  $(X, \cE)$ with $X = \p^1\times\p^1$ and $\cE=(E_1,E_2,E_3)$, where
$$
E_1=E_2 = H^0(X, \cO(k,0)), \quad E_2=H^0(X, \cO(0,l)), \text{ for } k,l>0.
$$
The collection $\cE$ is overdetermined, injective but not essential, since $\dd(\cE) = -1$ and $\Def(E_1,E_2) = -1$. The system 
\begin{equation}\label{eq2}
    s_1=s_2=s_3=0 \quad s_i\in E_i
\end{equation} 
has a root on $X$ if and only if 
\begin{equation}\label{eq3}
    s_1=s_2=0  \quad s_i\in E_i
\end{equation} 
has a root. The zero set of generic consistent system~(\ref{eq3}) is $p\times\p^1$ and therefore, the zero set of generic consistent system~(\ref{eq2}) is $l$ distinct points. In particular, if $l>1$, the generic non-empty zero set of system~(\ref{eq2}) is not a single fiber of the Kodaira map as $\cE$ is an injective system. But if $l=1$ the generic consistent system~(\ref{eq2}) has one root, so its zero set is a single fiber of Kodaira map. 

In the toric case, the sufficient and necessary condition for generic non empty zero set to be a single fiber of the Kodaira map was obtained in \cite[Theorem 2.2]{Mon19}. This result is based on classification of systems of Laurent polynomials which have only one root obtained in \cite{EGu}.
\end{Rem}

 \subsection{Generic non-empty zero set}  In this subsection we study the generic non-empty zero set of a system of equations $s_1=\ldots=s_k=0$ with $s_i\in E_i$. First let us summarize results of the last two sections on the generic non-empty zero set $Z_{\bf s}$. 
 
 \begin{Theorem}\label{maingen}
  Let $\cE$ be an overdetermined  collection of linear series on an irreducible variety $X$, with the essential subcollection $J$. Then for the generic consistent system ${\bf s}\in R_\cE$, the zero set $Z_{\bf s}$ is the generic zero set of the collection $\cE_{J^c}=(E_i)_{i\notin J}$ restricted to a fiber of a Kodaira map $\Phi_J$.
 \end{Theorem}
 \begin{proof}
 By Theorem~\ref{indep} and Proposition~\ref{esszero} the zero set of a generic consistent system ${\bf s} \in R_\cE$ is the zero set of a generic system $\cE_{J^c}=(E_i)_{i\notin J}$ restricted to a unique fiber of the Kodaira map $\Phi_J$. Moreover, such a restriction is generically consistent by Proposition~\ref{solv}.
 \end{proof}
 
 Theorem~\ref{maingen} expresses the  zero set of a system generic in the space of consistent systems defined by the collection $\cE$ as the zero set of the system which is generic in the space of {\em all systems} defined by collection $\cE_{J^c}$ restricted to a fiber of $\Phi_J$. We will use this result in coming sections to reduce questions about topology of generic non-empty zero set to questions about  topology of generic zero set.

 \begin{Proposition}\label{SBN}
 Let $f:X\to Y$ be a morphism between two algebraic varieties with $X$ - smooth. Then, the generic fiber is smooth.
 \end{Proposition}
 \begin{proof}
 Since the condition is local in $Y$ we can assume $Y$ to be affine. By the Noether normalization lemma there exists a finite morphism $g:Y\to \C^k$, with $k=\dim Y$. By the Bertini theorem the generic fiber of the composition $g\circ f: X\to \C^k$ is smooth. But since the generic fiber of $g\circ f$ is a finite union of disjoint fibers of $f$, the generic fiber of $f$ is also smooth.
 \end{proof}
 
 \begin{Theorem}
Let $X$ be smooth algebraic variety and let $E_1,\ldots, E_k$ be base-point free linear systems on $X$. Then for generic consistent $k$-tuple ${\bf s} =(s_1,\ldots,s_k)\in R_E$, the zero set $Z_{\bf s}$ is smooth. Moreover, the arithmetic genus of $Z_{\bf s}$ is constant for a generic choice of $\bf s$.
 \end{Theorem}
 
 \begin{proof}
 Let $\pi_1, \pi_2$ be two natural projections from $X\times\E$ to $X,\E$ respectively. Denote by $\pi_1, \pi_2$  also their restrictions to a incidence variety $\widetilde R_E$. For a system ${\bf s}\in R_E$ the zero set $Z_{\bf s}$ is given by $\pi_1(\pi_2^{-1}({\bf s}))$, in particular is isomorphic to $\pi_2^{-1}({\bf s})$. But since $\widetilde R_E$ is smooth ($\widetilde R_E$  is a vector bundle over $X$), the fiber of $\pi_2$ over the generic point ${\bf s} \in R_E$ is smooth  by Proposition~\ref{SBN}, and therefore such is the generic zero set $Z_{\bf s}$.
 
 For the second part notice that any algebraic morphism to an irreducible variety is flat over Zariski open subset. So for some Zariski open $U\subset  R_E$ the projection $\pi_2^{-1}(U)\to U$ is a flat family. The statement then follows from a fact that arithmetic genus is constant in flat families.
 \end{proof}

\subsection{Reduction theorem}
In this subsection we will formulate and prove the Reduction theorem. First we define what does it mean for two collections of linear systems to be equivalent.

\begin{Definition}\label{equiv}
Two collections $E_1,\ldots,E_k$ and $W_1,\ldots,W_l$ of linear systems on a quasi-projective irreducible variety $X$ are  called \emph{quivalent} if there exist Zariski open subsets $U\subset R_\E$ and $V\subset R_\W$ such that for any $u\in U$  there exists $v\in V$ (and for any $v\in V$ there $u\in U$) such that the zero sets $X_u$ and $X_v$ coincide.
\end{Definition}
 
\begin{Theorem}[Reduction theorem]
Any collection $\cE=(E_1,\ldots,E_k)$ of generically inconsistent linear series is equivalent to some collection $\cW$ of minimal defect $-1$. Moreover, if $\dd (\cE)= -d$ and $\cE_J=(E_1,\ldots,E_r)$ is an essential subcollection of $\cE$, then $\cW$ can be defined as
$$
W_1=\ldots=W_{r-d+1}=E_J,\quad W_{r-d+2}=E_{r+1},\ldots, W_{k-d+1}=E_k.
$$
\end{Theorem}

 \begin{proof}
First note that collections $\cE_J=(E_1,\ldots,E_r)$ and $\cW_K=(W_1,\ldots, W_{r-d+1})$ are equivalent collections which are the essential subcollections of $\cE$ and $\cW$ respectively. Indeed, since both collections are essential, by Proposition~\ref{esszero} they are equivalent if and only if generic fibres of their Kodaira maps  coincide. But this follows directly from the fact that $W_1=\ldots=W_{r-d+1}=E_J=E_1\cdot\ldots \cdot E_r$.
 
Therefore, we have two collections $\cE$ and $\cW$ with equivalent essential subcollections $J=(E_1,\ldots,E_r)$ and $K=(W_1,\ldots, W_{r-d+1})$ and coinciding complements: 
$$
\cE_{J^c}= (E_{r+1},\ldots, E_k) = (W_{r-d+2},\ldots, W_{k-d+1})= \cW_{K^c}.
$$
But any two such collections are equivalent since the zero set $Z_{\bf u}$ of a generic system is a solution of the system compliment to the essential subsystem restricted to the zero set of the essential subsystem.
 \end{proof}

\section{Equivariant linear systems on homogeneous varieties}\label{hom}

This section is devoted to the study of $G$-invariant linear systems on a  complex variety $X$ with a transitive $G$-action. First, we work with general homogeneous space and prove Theorem~\ref{mainhom} which reduces the study of generic non-empty zero sets of overdetermined  systems to the study of generic complete intersections.

We apply then this result to obtain Theorem~\ref{mainsph} which together with results of \cite{KK16} provides a strategy for computation of discrete invariants of generic non-empty zero set of a system of equations associated to a collection of overdetermined  linear series on a spherical homogeneous space.

\subsection{Linear systems on homogeneous varieties}
In this subsection we will study some general results on linear systems on homogeneous varieties. The main result of this subsection is a reduction of an overdetermined  collection of linear series to several isomorphic generically consistent collections. 

Let $G$ be a connected algebraic group, and $X=G\/H$ be a $G$-homogeneous space. Let us denote by $x_0\in X$ the class of identity element $e\cdot H \in G\/H$. Let $\cL_1, \ldots, \cL_k$ be globally generated $G$-linearized line bundles on $G\/H$. For each $i=1, \ldots, k$ let $E_i$ be a nonzero $G$-invariant linear system for $\cL_i$ i.e. $E_i$ is a finite dimensional $G$-invariant subspace of $H^0(X, \cL_i)$.

Each $E_J$ is $G$-invariant and hence it is base-point free on $G\/H$. Thus the Kodaira map $\Phi_J$ is defined on the whole $G\/H$.  Since $E_1,\ldots,E_k$ are $G$-invariant, $E_J^*$ is a linear representation of $G$ for any $J\subset\{1,\ldots,k\}$ and, therefore, there is a natural action of $G$ on $\p(E_J^*)$. It is easy to see that the Kodaira map $\Phi_J$ is equivariant for this action. Therefore, the image $\Phi_J(X)$ is a quasi-projective homogeneous $G$-variety isomorphic to $G\/(G_{\Phi_J(x_0)})$, where $G_{\Phi_J(x_0)}$ is a stabilizer of  $\Phi_J(x_0)\in \p(E_J^*)$. For $J\subset \{1,\ldots,k\}$ we will denote the stabilizer $G_{\Phi_J(x_0)}$ by $\Gamma_J\subset G$.

\begin{Definition} 
Two collections of $G$-invariant linear systems $\cE=(E_1,\ldots,E_k)$, $\cE'=(E_1',\ldots E_k')$ on homogeneous spaces $X, X'$ respectively are isomorphic if there exists an $G$ equivariant isomorphism $f: X\to X'$ and an isomorphism of $G$-linearized line bundles $\phi_i:\cL_i \to f^*\cL_i'$ for any $i=1,\ldots, k$ such that $\phi_i^*\circ f^*(E_i')=E_i$.
\end{Definition}

\begin{Proposition}\label{fibers}
Let $\Phi_J:X\to \p(E_J^*)$ and $\Gamma_J$ be as before then:\\
(i) for any $y\in\Phi_J(X)$ the fiber $F_y:= \Phi_J^{-1}(y)$ has a structure of $\Gamma_J$-variety;\\
(ii) any fiber $F_y$ is isomorphic to $F_{y_0}\cong \Gamma_J\/H$ as $\Gamma_J$-variety;\\
(iii) for any $G$ equivariant linear system $V$ on $X$ and any point  $y\in\Phi_J(X)$ the restriction $E|_{F_y}$ is $\Gamma_J$-invariant. Moreover, a pair $F_y, V\big|_{F_y}$ is isomorphic to the pair $F_{y_0}, V\big|_{F_{y_0}}$.
\end{Proposition}

\begin{proof}
For parts $(i)$ and $(ii)$ let $G_y$ be the stabilizer of a point  $y\in \Phi_J(X)$, then the fiber $F_y$ is an homogeneous $G_y$ variety. But also $G_y$ is conjugate to $\Gamma_J$, i.e. $G_y=g \Gamma_J g^{-1}$ for some $g\in G$.  Then an equivariant isomorphism between $F_{y_0}$ and $F_y$ (which also will define $\Gamma_J$ on $F_y$) can be defined by
$$
\Gamma_J\times F_{y_0} \to G_y \times F_y, \quad (\gamma,p)\mapsto (g \gamma g^{-1} , gp).
$$

Part $(iii)$ follows directly from the construction above and the fact that linear system $V$ is $G$-invariant.
\end{proof}

By Proposition~\ref{fibers}, any fiber of the Kodaira map associated to a $G$-invariant linear series is a homogeneous variety and in particular is smooth. Therefore, irreducible components of fibers coincide with connected components. Next proposition is a more precise version of Proposition~\ref{fibers}, which deals with connected components of fibers of the Kodaira map.

\begin{Proposition}\label{concompfibre}
Let $\Phi_J:X\to \p(E_J^*)$ and $\Gamma_J$ be as before then:\\
(i) connected components of fibres of  $\Phi_J$ have a structure of $\Gamma_J^0$-variety;\\
(ii) any two connected components of any two fibres are isomorphic as $\Gamma_J^0$-varieties and, in particular are isomorphic to $\Gamma_J^0\/(\Gamma_J^0\cap H)$, where $\Gamma_J^0$ is the connected component of identity in $\Gamma_J$;\\
(iii) for any $G$ equivariant linear system $V$ on $X$ and two connected components  $C_1, C_2$ of any two fibres, the restrictions $E\big|_{C_1}, E\big|_{C_2}$ are $\Gamma_J^0$-invariant. Moreover,  pairs $C_1, V\big|_{C_1}$ and $C_2, V\big|_{C_2}$ are isomorphic.
\end{Proposition}

\begin{proof}
Most of the proof  is absolutely analogous to the proof of Proposition~\ref{fibers}. The only statement which needs clarification is that any connected component of a fiber is isomorphic to $\Gamma_J^0\/(\Gamma_J^0\cap H)$.  By Proposition~\ref{fibers}  it is enough to check this for a connected component of a given fiber, say $F_{y_0}$. The rest easily follows from the fact that $F_{y_0}= \Gamma_J\/H$. 
\end{proof}

For a subcollection $J\subset \{1,\ldots, k \}$ we will call the number of connected components of a fiber of the Kodaira map $\Phi_J:G\/H\to \p(E_J^*)$ {\em the index} of $J$ and denote it by $ind(J)$. One can describe $ind(J)$ in a group theoretic way, this description clarifies the term ``index''.

Connected components of identity $\Gamma_J^0, H^0$ are normal subgroups of groups $\Gamma_J, H$ respectively. There exists a natural homomorphism between $i:H\/H^0\to\Gamma_J\/\Gamma_J^0$ induced by the inclusion of $H\subset \Gamma_J$. The homomorphism $i$ is well-defined since $H^0$ is a subgroup of $\Gamma_J^0$. It is easy to see that connected components of $ \Gamma_J\/H$ and hence of any other fiber of $\Phi_J$ are in one to one correspondence with elements of the coset set 
$$
(\Gamma_J\/\Gamma_J^0)\Big/i(H\/H^0),
$$
i.e. $ind(J)$ is equal to the index of $i(H\/H^0)$ in $\Gamma_J\/\Gamma_J^0$.

\begin{Theorem}\label{mainhom}
Let $X=G\/H$ be a homogeneous space and let $E_1,\ldots,E_k$ be $G$ - invariant linear systems on $X$, with the essential subcollection $J\subset\{1,\ldots,k\}$. Then the zero set $X_{\bf s}$ for generic ${\bf s}\in R_\cE$ is a disjoint union of $ind(J)$ subvarieties $Y_1,\ldots, Y_{ind(J)}$. Moreover, for any $i=1,\ldots, ind(J)$ $Y_i$ is a generic zero set of a collection of linear series isomorphic to $\cE_{J^c}=(E_i)_{i\notin J}$ restricted to $\Gamma_J^0\/(\Gamma_J^0\cap H)$.
\end{Theorem}

\begin{proof}
By Theorem~\ref{maingen} the zero set $Z_{\bf s}$ of a generic consistent system given by ${\bf s}\in R_\cE$ is the zero set of a generic system $\cE_{J^c}$ restricted to the fiber of the Kodaira map $\Phi_J$. But by Proposition~\ref{concompfibre} connected components with the restrictions of the collection $\cE_{J^c}$ to them are isomorphic.
\end{proof}


\subsection{Linear series in spherical varieties} \label{Sphere}  In this subsection we will study $G$-invariant linear series on spherical varieties. A homogeneous space $G\/H$ of a reductive group $G$ is called {\it spherical} if some (and hence any) Borel subgroup $B\subset G$ has an open dense orbit in $G\/H$. Starting from this point a group $G$ will assumed to be reductive and a homogeneous space $G\/H$  will assumed to be spherical.

Any $G$-invariant linear series $V$ on $G\/H$ is a representation of a reductive group $G$, therefore it is a direct sum of irreducible representations:
$$
V= \bigoplus_\lambda V_\lambda.
$$
It is well known that the decomposition above is multiplicity free, that is each irreducible representation appears at most once in it. Indeed, let $V_\lambda$ and $V'_\lambda$ be two different irreducible representations with the same highest weight appearing in decomposition of $V$, and let $s$ and $s'$ be highest weight vectors in $V_\lambda$ and $V'_\lambda$ respectively. In particular, both $s$ and $s'$ are $B$-eigensections of weight $\lambda$ of some $G$-linearised line bundle $\cL$. Therefore the ratio $s/s'$ is a $B$-invariant rational function on $G\/H$. Since $X$ has an open $B$-orbit we conclude that $s/s'$ is constant, so $V_\lambda = V'_\lambda$. The set of weights appearing in the decomposition of $V$ is called {\em $G$-spectrum} of $V$ and is denoted by $\Spec_G(V)$. The pair of a $G$-linearized line bundle $\cL$ and a finite subset $A$ of $\Spec_G(H^0(X,\cL))$ determines uniquely a $G$-invariant linear series on $X$.

The main result of this subsection is Theorem~\ref{mainsph} which realizes a generic non-empty zero set defined by overdetermined  collection of linear series, as a zero set, defined by generically consistent collection. In \cite{KK16}, for a collection of linear series $\cE=(E_1,\ldots,E_k)$ and for a generic choice of ${\bf s}\in \E$, some discrete invariants of the zero set $Z_{\bf s}$ were computed in terms of combinatorics of the Newton-Okounkov polytope. The Newton-Okounkov polytope is constructed as a polytope fibered over moment polytope with string polytopes as fibers. The construction of Newton-Okounkov polytope depends only on $G$-spectra of linear series $E^k$ (for more details see \cite{KK16}). These results together with Theorem~\ref{mainsph} provide a strategy for computing discrete invariants of generic non-empty zero set defined by overdetermined  collection of linear series.

\begin{Lemma}\label{stillsph}
Let $H$ be a spherical subgroup of a reductive group $G$, let $K$ be a connected reductive subgroup of $G$ which contains $H$ (i.e $H\subset K\subset G$). Then $H^0$ is a spherical subgroup of $K$.
\end{Lemma}
\begin{proof}
First notice that if $H$ is spherical subgroup of $G$ then such is $H^0$ so without loss of generality we can assume that $H=H^0$.

Now consider a point  $x\in K\/H\subset G\/H$, there exists a Borel subgroup $B$ of $G$ such that $B\cdot x$ is a dense in $G\/H$.  Let $U$ be the intersection of the dense orbit $B\cdot x$ with $K\/H$, in other words $U=(B\cap K)\cdot x$. Note that $(B\cap K)$ is consistent and therefore $(B\cap K)^0$ is contained in some Borel subgroup $B_K$ of $K$. From the other hand, since $B\cdot x$  is open in $G\/H$, $U$ is open and dense in $K\/H$, and since $K\/H$ is irreducible the orbit $(B\cap K)^0\cdot x$ is dense in $K\/H$. We conclude that the orbit $B_K\cdot x$ is dense in $K\/H$ as $B_K$ contain $(B\cap K)^0$.
\end{proof}

\begin{Proposition}\label{sphfiber}
Let $E_1,\ldots, E_k$ be $G$-invariant linear systems on a spherical homogeneous space $G\/H$ and $\Gamma_J$ be a reductive group for some $J\subset \{1,\ldots,k\}$. Then the connected components of fibers of the Kodaira map $\Phi_J$ are isomorphic spherical homogeneous spaces. Moreover, for any  $G$-invariant linear series $V$,  all restrictions of $V$ to connected components of fibers are isomorphic.
\end{Proposition}

\begin{proof}
By part $(ii)$ of Proposition~\ref{concompfibre} connected components of the fibers of the Kodaira map $\Phi_J$ are  homogeneous spaces which are isomorphic to $\Gamma_J^0\/(\Gamma_J^0\cap H)$, and since $\Gamma_J^0$ is reductive  by Lemma~\ref{stillsph} they are also spherical. The last statement of the above proposition is identical to part $(iii)$ of Proposition~\ref{concompfibre}. 
\end{proof}

\begin{Theorem}\label{mainsph}
Let $\cE=(E_1,\ldots, E_k)$ and $G\/H$ be as before. Let also $J$ be the essential subcollection of $\cE$ such that $\Gamma_J$ is a reductive group. Then for the generic system ${\bf s}\in R_\E$ the zero set $Z_{\bf s}$ is a disjoint union of $ind(J)$ varieties $Y_1,\ldots, Y_{ind(J)}$. Moreover, subvarieties $Y_j$'s are defined by isomorphic $\Gamma_J^0$-invariant collections of linear series on the spherical variety $\Gamma_J^0\/(\Gamma_J^0\cap H)$.
\end{Theorem}
\begin{proof}
By Theorem~\ref{mainhom} the zero set $Z_{\bf s}$ of a generic consistent system ${\bf s}\in R_\cE$ is a union of $ind(J)$ subvarieties $Y_1,\ldots, Y_{ind(J)}$, with $Y_j$'s  defined by isomorphic $\Gamma_J^0$-invariant collections of linear series on connected components of Kodaira map $\Phi_J$. But connected components of Kodaira map $\Phi_J$ are spherical and isomorphic to $\Gamma_J^0\/(\Gamma_J^0\cap H)$ by Proposition~\ref{sphfiber}.
\end{proof}

\begin{Corollary}
In the situation above, the arithmetic genus $g(Z_{\bf s})$ of the generic non-empty zero set  $Z_{\bf s}$ could be computed as
$$
g(Z_{\bf s})= ind(J)g(Y_i),
$$
for any $i=1,\ldots, ind(J)$, and can be computed in terms combinatorics of Newton-Okounkov polytopes. If in addition all linear series from collection $\cE_{J^c}$ restricted to $\Gamma_J^0\/(\Gamma_J^0\cap H)$  are injective, mixed Hodge numbers $h^{p,0}(Z_{\bf s})$  of the generic non-empty zero set  $Z_{\bf s}$ can be computed as
$$
 h^{p,0}(Z_{\bf s})=ind(J)  h^{p,0}(Y_i),
$$
for any $i=1,\ldots, ind(J)$, and can be computed in terms combinatorics of Newton-Okounkov polytopes.
\end{Corollary}
\begin{proof}
Corollary follows immediately from Theorem~\ref{mainsph} and Theorems 1 and 2 of \cite{KK16}.
\end{proof}

Theorem~\ref{mainsph} involves condition on $\Gamma_J$ to be reductive, the following theorem provides a geometric criterion for a subgroup of a reductive group to be reductive.

\begin{Theorem}[\cite{Tim}]\label{reductive}
 Let $G$ be a reductive algebraic group then a closed subgroup $H$ of $G$ is reductive if and only if the coset space $G\/H$ is an affine algebraic variety.
\end{Theorem}

\begin{Rem}
The condition on $\Gamma_J$ to be reductive is quite restrictive. However, for any triple $H\subset K\subset G$ where $H$ is a spherical subgroup of $G$ and $K$ is reductive, one can realize $K$ as $\Gamma_J$ for some collection of $G$-invariant linear series on $G/H$. Indeed, let $\pi:G/H \to G/K$ be natural projection and let $\cE=(E_1,\ldots, E_k)$ be an injective (i.e. such that  $E_J$ is very ample) essential collection of  $G$-invariant linear system on $G/K$. Then the pullback collection $\pi^*\cE=(\pi^*E_1,\ldots,\pi^* E_k)$ is an essential collection of linear series on $G/H$ with $\Gamma_J=K$. The classification of all such triples $H\subset K\subset G$ follows from several works. Following ideas from \cite{Lun} in \cite{Los} spherical subgroups were classified. The classification of spherical reductive subgroups was obtained in  \cite{RedSph}. And finally, the containment relation between spherical subgroups  was obtained in \cite{Hof18}. 
\end{Rem}

 \subsection{Example}\label{ex}

In this subsection we will give a concrete example of an application of Theorem~\ref{mainsph}. We will work with homogeneous space $\GL_n/U$. Let $\cE=(E_1,E_2,E_3)$ be a collection of linear series with
$$
E_1=E_2= Span(c, {\det}^k),
$$
where $c$ is a constant function, ${\det}^k(g\cdot U)= \det(g)^k$, and $E_3$ is a very ample linear series.

The minimal defect of $\cE$ is $-1$, the essential subcollection is $J=\{1, 2\}$, and $\Gamma_J=\SL_n[k]$, where 
$$
\SL_n[k]=\left\{ g\in \GL_n\,| \,det(g)^k=1 \right\}.
$$
Therefore, a connected component of a fiber of the Kodaira map $\Phi_J$ is isomorphic to $\SL_n/U$ and the number of connected components of a fiber is $ind(J)=k$. 

It follows by Theorem~\ref{mainsph} that the generic non-empty zero set $Z_{\bf s}$ of a system $s_1=s_2=s_3=0$ is a union of $k$ subvarieties $Y_1,\ldots, Y_k$ such that each of $Y_i$ is a hypersurface in $\SL_n/U$ cut out by a generic section $s\in E_3|_{\SL_n/U}$. In particular geometric genus and mixed Hodge numbers $h^{p,0}$ of $Z_{\bf s}$ could be computed using results of \cite{KK16}.

\section{Linear series on Algebraic Torus}\label{torus}

In this section we study equivariant linear series on $(\C^*)^n$. The results of this section were previously published in~\cite{Mon17} and included here for completeness of exposition. We will start with some notations and definitions. With a Laurent polynomial $f$  in $n$ variables one can associate its support $supp(f)\subset \Z^n$ which is the set of exponents of monomials having non-zero coefficient in $f$ and its Newton polyhedra $\Delta(f)\subset \R^n$ which is the convex hull of $supp(f)$ in $\R^n$.  For a finite set $A\subset \Z^n$, let $E_A$ be a vector space of Laurent polynomials with support in $A$. It is easy to see that any equivariant linear series on $(\C^*)^n$ is of the form $E_A$ for some $A$ and that $A$ is a $(\C^*)^n$-spectrum of $E_A$.

Any connected algebraic subgroup of $(\C^*)^n$ is an algebraic torus, in particular such is group $\Gamma_J^0$. Therefore, Theorem~\ref{mainsph} has a much nicer form in the case of equivariant linear series on $(\C^*)^n$.
 
\begin{Theorem}\label{maintoric}
Let $\cA=(A_1,\ldots,A_{k})$ be a collection of finite subsets of $\Z^n$ and $E_i=E_{A_i}$ be corresponding linear systems on $(\C^*)^n$. Let $J$ the essential subcollection of $E_1,\ldots,E_k$. Then for the generic consistent system ${\bf s}\in R_\cE$ the zero set $Z_{\bf s}$ is a disjoint union of $ind(J)$ subvarieties each of which is defined by a generic system given by isomorphic collections of linear systems on algebraic torus.
\end{Theorem}

In the toric case Theorem~\ref{maintoric} can be formulated much more concretely. This will allow us to find discrete invariants of a generic non-empty zero set explicitly. We give an example of such result in Theorem~\ref{BKK}. To state a concrete version of Theorem~\ref{maintoric} we would need more notations. For a collection $\cA=(A_1,\ldots, A_k)$ of finite subsets of $\Z^n$ and subcollection $J$ let 

\begin{itemize}
\item $A_J$ be the Minkowski sum $\sum_{i\in J}A_i$;
\item $L(J)$ be a vector subspace of $\R^n$ parallel to the affine span of $A_J$ and $\pi_J:\R^n\to \R^n/L(J)$ be the natural projection;
\item $\Lambda(J) = L(J)\cap\Z^n$ the lattice of integral points in $L(J)$;
\item $G_J$ the group generated by all the differences of the form $(a-b)$ with $a,b\in A_i$ for any $i\in J$;
\item $ind(G_J)$ the index of $G_J$ in $\Lambda(J)$.
\end{itemize}

\begin{Theorem}\label{toric}
 Let $\cA=(A_1,\ldots,A_{k})$ be a collection of finite subsets of $\Z^n$ and $E_i=E_{A_i}$ be corresponding linear systems on $(\C^*)^n$, let also $J$ the essential subcollection of $E_1,\ldots,E_k$. Then\\
 (i) The defect of a collection of linear series $\Def(J)$ (in the sense of Definition~\ref{defK}) is equal to the defect of the vector subspaces $\Def(L(A_i))_{i\in J}$ (in the sense of Definition~\ref{defc});\\
 (ii) The number $ind(J)$ of connected components of a fiber of Kodaira map $\Phi_J$ is equal to the index $ind(G_J)$;\\
 (iii) for the generic consistent system ${\bf s}\in R_\cE$ the zero set $Z_{\bf s}$ is a disjoint union of $ind(J)$  subvarieties each of which is given by a generic system with the same supports $(\pi_J(\cA_i))_{i \notin J}$.
\end{Theorem}
\begin{proof}
For the proof see Section 4 of \cite{Mon17} and in particular proof of Theorem 18.
 \end{proof}
 
 The following Theorem is an example of the application of Theorem~\ref{toric}.
 \begin{Theorem}[\cite{Mon17} Theorem 20]\label{BKK}
Let $\cA_1,\ldots,\cA_{n+k}\subset \Z^n$ be such that $d(\cA)=-k$ and $J$ be the unique essential subcollection. Then the zero set $Y_{\bf f}$ of the generic consistent system has dimension 0, and the number of points in $Y_{\bf f}$ is equal to
 $$
 (n-\#J + k)! \cdot ind(J) \cdot Vol(\pi_J(\Delta_i)_{i\notin J}),
 $$
 where $\Delta_i$ is the convex hull of $\cA_i$ and $Vol$ is the mixed volume on $\R^n/L(J)$ normalized with respect to the lattice $\Z^n/\Lambda(J)$.
\end{Theorem}

If $k=0$ this theorem coincides with the BKK Theorem. In the case $k=1$ the generic number of solution appears as the corresponding degree of sparse resultants and was computed in \cite{D'AS}. In a similar fashion, Theorem~\ref{toric} could be applied to the computation of any other discrete invariants which can be computed by means of Newton polyhedra theory.

 \bibliographystyle{alpha}
\bibliography{thesis}
 \end{document}